\documentclass[12pt,a4paper]{article}

\usepackage{amsfonts}
\usepackage{amsmath}
\usepackage{amssymb}
\usepackage{amstext}
\usepackage{amsthm}   
\usepackage{mathrsfs} 

\newcommand{\R}{\mathbb R}

\newcommand{\Z}{\mathbb Z}

\newcommand{\eps}{\varepsilon}

\newtheorem{teo}{Theorem}

\newtheorem{lemma}{Lemma}

\newtheorem{theorem}[teo]{Theorem}
\newtheorem{defin}[teo]{Definition}

\theoremstyle{remark}
\newtheorem{rem}[teo]{Remark}

\DeclareMathOperator{\Id}{Id}
\DeclareMathOperator{\cat}{cat}
\DeclareMathOperator{\dist}{dist}

\begin{document}
\title{On the number of nodal solutions for a nonlinear elliptic problem on symmetric
Riemannian manifolds}
\author{M.Ghimenti\thanks{Dipartimento di Matematica Applicata,
Universit\`a di Pisa, via Buonarroti 1c, 56100, Pisa, Italy, e-mail
{\tt ghimenti@mail.dm.unipi.it, a.micheletti@dma.it} },
A.M.Micheletti\addtocounter{footnote}{-1}\footnotemark }
\date{}
\maketitle

\begin{abstract}
\noindent We consider the problem
$-\eps^2\Delta_g u+u=|u|^{p-2}u$ in $M$, where $(M,g)$ is a symmetric
Riemannian manifold. We give a multiplicity result for
antisymmetric changing sign solutions.

\noindent\textbf{Keywords:} Riemannian Manifolds, Nodal Solutions,
Topological Methods

\noindent\textbf{Mathematics Subject Classification:}  35J60, 58G03
\end{abstract}
\section{Introduction}
Let $(M,g)$ be a smooth connected compact Riemannian manifold of finite dimension
$n\geq 2$ embedded in $\R^N$.
We consider the problem
\begin{equation}\label{P}\tag{${\mathscr P}$}
-\eps^2\Delta_g u+u=|u|^{p-2}u\text{ in } M, \ \
u\in H^1_g(M)
\end{equation}
where $2<p<2*=\frac {2N}{N-2}$, if $N\geq 3$.

Here $H^1_g(M)$ is the completion of $C^\infty(M)$ with respect to
\begin{equation}
||u||^2:g=\int_M |\nabla_g u|^2+u^2d\mu_g
\end{equation}

It is well known that the problem (\ref{P}) has a mountain pass solution $u_\eps$.
In \cite{BP05} the authors showed that $u_\eps$ has a spike layer and its peak point
converges to the maximum point of the scalar curvature of $M$ as $\eps$ goes to 0.

Recently there have been some results on the influence of the topology and
the geometry of $M$ on the number of solutions of the problem.
In \cite{BBM07} the authors proved that, if $M$ has a rich topology, problem (\ref{P}) has
multiple solutions. More precisely they show that problem (\ref{P}) has at least
$\cat(M)+1$ positive nontrivial solutions for $\eps$ small enough. Here $\cat(M)$
is the Lusternik-Schnirelmann category of $M$.
In \cite{Vta} there is the same result for a more general nonlinearity. Furthermore in \cite{Hta}
it was shown that the number of solution is influenced by the topology of a suitable
subset of $M$ depending on the geometry of $M$. To point out the role of the geometry
in finding solutions
of problem (\ref{P}), in \cite{MP2ta}
it was shown that for any stable critical point of the scalar curvature
it is possible to build positive single peak solutions. The peak of these solutions approaches
such a critical point as $\eps$ goes to zero.

Successively in \cite{DMP} the authors build positive $k$-peak solutions whose peaks collapse
to an isolated local minimum point of the scalar curvature as $\eps$ goes to zero.

The first result on sign changing solution is in \cite{MPta} where it is showed the
existence of a solution
with one positive peak $\eta_1^\eps$ and one negative peak $\eta_2^\eps$ such that, as $\eps$ goes to zero,
the scalar curvature $S_g(\eta_1^\eps)$ (respectively $S_g(\eta_2^\eps)$)
goes to the minimum (resp. maximum)
of the scalar curvature when the scalar curvature of $(M,g)$ is non constant.
Here we give a multiplicity result for changing sign solutions when the Riemannian
manifold $(M,g)$ is symmetric.

We look for solutions of the problem
\begin{equation}\label{Ptau}
\left\{
\begin{array}{ll}
\tag{$\mathscr P_\tau$}
-\eps^2\Delta_g u+u=|u|^{p-2}u&
u\in H^1_g(M);\\&\\
u(\tau x)=-u(x)  & \forall x\in M,
\end{array}
\right.
\end{equation}
where $\tau:\R^N\rightarrow\R^N$ is an
orthogonal linear
transformation  such that $\tau\neq\Id$,
$\tau^2=\Id$, $\Id$ being the identity of $\R^N$. Here $M$ is a compact connected Riemannian manifold
of dimension $n\geq 2$ and $M$ is a regular submanifold of $\R^N$ which is invariant with respect to $\tau$.
Let $M_\tau:=\{x\in M\ :\ \tau x=x\}$ be the set of the fixed points with respect to the
involution $\tau$; in the case $M_\tau\neq \emptyset$ we assume that $M_\tau$ is a regular
submanifold of $M$.

We obtain the following result.
\begin{theorem}\label{mainteo}
The problem \ref{Ptau} has at least $G_\tau-\cat(M-M_\tau)$ pairs of solutions $(u,-u)$
which change sign (exactly once) for $\eps$ small enough
\end{theorem}

Here
$G_\tau-\cat$ is the $G_\tau$-equivariant Lusternik Schnirelmann category for the group
$G_\tau=\{\Id,\tau\}$.

In \cite{CC03} the authors prove a result of this type for
the Dirichlet problem

\begin{equation}\label{Plambda}
\left\{
\begin{array}{ll}
\tag{$\mathscr P_\lambda$}
-\Delta u-\lambda u-|u|^{2^*-2}u=0& u\in H^1_0(\Omega);
\\
u(\tau x)=-u(x). &
\end{array}
\right.
\end{equation}
Here $\Omega$ is a bounded smooth domain invariant
with respect to $\tau$ and $\lambda$ is a positive parameter.

We point out that in the case of the unit sphere $S^{N-1}\subset \R^N$ (with the metric $g$ induced
by the metric of $\R^N$) the theorem of existence of changing sign solutions
of \cite{MPta} can not be used because it holds for manifold of non constant curvature. 
Instead, we can apply Theorem \ref{mainteo} to obtain sign changing solutions 
because we can consider
$\tau=-\Id$, and we have $G_\tau-\cat S^{N-1}=N$.

Equation like (\ref{P}) has been extensively studied in a flat bounded domain $\Omega\subset \R^N$.
In particular, we would like to compare problem (\ref{P}) with the following Neumann problem
\begin{equation}\label{pn}\tag{$\mathscr P_N$}
\left\{
\begin{array}{ll}
-\eps^2\Delta u +u=|u|^{p-2}u& \text{in }\Omega;\\
\frac{\partial u}{\partial \nu}=0&\text{ in }\partial\Omega.
\end{array}
\right.
\end{equation}
Here $\Omega$ is a smooth bounded domain of $\R^N$ and $\nu$ is the unit outer normal to $\Omega$.
Problems (\ref{P}) and (\ref{pn}) present many similarities. We recall some classical results
about the Neumann problem.

In the fundamental papers \cite{LNT,NT1,NT2}, Lin, Ni and Takagi established the existence of
least-energy solution to (\ref{pn}) and showed that for $\eps$ small enough the least energy
solution has a boundary spike, which approaches the maximum point of the mean curvature $H$ of
$\partial\Omega$, as $\eps$ goes to zero. Later, in \cite{DFW,W1} it was proved that for any stable
critical point of the mean curvature of the boundary it is possible to construct single
boundary spike layer solutions, while in \cite{G,WW,Li} the authors construct multiple boundary
spike solutions at multiple stable critical points of $H$. Finally, in \cite{DY,GWW} the authors
proved that for any integer $K$ there exists a boundary $K$-peaks solutions, whose peaks
collapse to a local minimum point of $H$.
\section{Setting}

We consider the functional defined on $H^1_g(M)$
\begin{equation}
J_\eps(u)=\frac1{\eps^N}\int_M\left(
\frac12 \eps^2|\nabla_gu|^2+\frac12 |u|^2-\frac1p|u|^p
\right)d\mu_g.
\end{equation}
It is well known that the
critical points of $J_\eps(u)$
constrained on the Nehari manifold
\begin{equation}
{\cal N}_\eps=\left\{
u\in H^1_g\smallsetminus \{0\}\ :\ J'_\eps(u)u=0 \right\}
\end{equation}
are non trivial solution of problem (\ref{P}).

The transformation $\tau:M\rightarrow M$ induces a transformation on $H^1_g$
we define the linear operator $\tau^*$ as

\begin{eqnarray*}
\tau^*&:&H^1_g(M)\rightarrow H^1_g(M)\\
&&\tau^*(u(x))=-u(\tau(x))
\end{eqnarray*}

and $\tau^*$ is a selfadjoint operator with respect to the scalar product on $H^1_g(M)$
\begin{equation}
\langle u,v\rangle_\eps=\frac1{\eps^N}\int_M\big(
 \eps^2\nabla_gu\cdot\nabla_gv+ u\cdot v \big) d\mu_g.
\end{equation}
Moreover $||\tau^*u||_{L^p(M)}=||u||_{L^p(M)}$, and
$||\tau^*u||_\eps=||u||_\eps$, thus $J_\eps(\tau^*u)=J_\eps(u)$.
Then, for the Palais principle, the nontrivial solutions of (\ref{Ptau}) are the critical points of the restriction
of $J_\eps$ to the $\tau$-invariant Nehari manifold
\begin{equation}
{\cal N}_\eps^\tau=\{u\in {\cal N}_\eps \ :\ \tau^*u=u\}={\cal N}_\eps\cap H^\tau.
\end{equation}
Here $H^\tau=\{u\in H^1_g\ :\ \tau^* u=u\}$.

In fact, since $J_\eps(\tau^* u)=J_\eps(u)$ and $\tau^*$ is a selfadjoint operator we have
\begin{equation}
\langle \nabla J_\eps(\tau^*u),\tau^*\varphi\rangle_\eps=
\langle \nabla J_\eps( u), \varphi\rangle_\eps\ \ \forall\varphi\in H^1_g(M).
\end{equation}
Then $\nabla J_\eps( u)=\tau^* \nabla J_\eps(\tau^*u)=\tau^* \nabla J_\eps(u)$ if $\tau^*u=u$.

We set
\begin{eqnarray}
&&m_\infty=\inf\limits_{\int_{\R^N} |\nabla u|^2+u^2=\int_{\R^N}|u|^p}
\frac 12 \int_{\R^N} |\nabla u|^2+u^2-\frac 1p\int_{\R^N}|u|^p;\\
\nonumber &&\\
&&m_\eps=\inf\limits_{u\in {\cal N}_\eps}J_\eps;\\
&&m_\eps^\tau=\inf\limits_{u\in {\cal N}_\eps^\tau}J_\eps.
\end{eqnarray}
\begin{rem}\label{remps}
It is easy to verify that $J_\eps$ satisfies the Palais Smale condition on
${\cal N}_\eps^\tau$.
Then there exists $v_\eps$ minimizer of $m_\eps^\tau$ and $v_\eps$ is a critical point for
$J_\eps$ on $H^1_g(M)$. Thus $v_\eps^+$ and $v_\eps^-$
belong to ${\cal N}_\eps$, then $J_\eps(v_\eps)\geq 2m_\eps$.
\end{rem}

We recall some facts about equivariant Lusternik-Schnirelmann theory. If $G$ is a compact Lie
group, then a $G$-space is a topological space $X$ with a continuous $G$-action
$G \times X    \rightarrow X$,
$(g, x)\mapsto gx$. A $G$-map is a continuous function $f : X  \rightarrow Y$
between $G$-spaces $X$ and $Y$
which is compatible with the $G$-actions, i.e. $f (gx) = gf (x)$ for all $x \in X$,
$g \in G$. Two
$G$-maps $f_0$, $f_1 : X \rightarrow Y$ are $G$-homotopic if there is a homotopy
$\theta : X \times [0, 1] \rightarrow Y$ such
that
$\theta(x, 0) = f_0(x)$,
$\theta(x, 1) = f_1(x)$ and
$\theta(gx, t) = g\theta (x, t)$ for all$x\in X$, $g \in G$,
$t \in [0, 1]$. A subset $A$ of a $X$ is $G$-invariant if $ga \in A$ for every $a\in A$,
$g \in G$. The $G$-orbit of a point $x \in X$ is the set $Gx = \{gx : g \in G\}$.
\begin{defin}
The $G$-category of a $G$-map $f : X \rightarrow Y$ is the smallest number $k = G-\cat(f)$
of open $G$-invariant subsets $X_1, \dots , X_k$ of $X$ which cover $X$ and which have the
property that, for each $i = 1,\dots , k$, there is a point $y_i \in Y$ and a $G$-map
$\alpha_i : X_i \rightarrow Gy_i\subset Y$
such that the restriction of $f$ to $X_i$ is $G$-homotopic to $\alpha_i$.
If no such covering exists we define
$G-\cat(f)=\infty$.
\end{defin}
In our applications, $G$ will be the group with two elements, acting as $G_\tau = \{\Id, \tau\}$
on $\Omega$, and as $\Z/2 = \{1,-1\}$ by multiplication on the Nehari manifold
${\cal N}_\eps^\tau$.
We remark the following result on the equivariant category.
\begin{teo}\label{castroclapp}
Let $\phi : M \rightarrow \R$ be an even $C^1$ functional on a complete $C^{1,1}$ submanifold $M$
of a Banach space which is symmetric with respect to the origin. Assume that $\phi$ is bounded
below and satisfies the Palais Smale condition $(PS)_c$ for every $c\leq d$.
Then $\phi$ has at least $\Z/2-\cat(\phi^d)$
antipodal pairs $\{u,-u\}$ of critical points with critical values $\phi(\pm u)\leq d$.
\end{teo}
\section{Sketch of the proof}
In our case we consider the even positive $C^2$ functional $J_\eps$ on the $C^2$
Nehari manifold ${\cal N}_\eps^\tau$ which is symmetric with respect to the origin.
As claimed in Remark \ref{remps}, $J_\eps$ satisfies Palais Smale
condition on ${\cal N}_\eps^\tau$. Then we can apply
Theorem \ref{castroclapp} and our aim is to get an estimate of this lower bound for the number of solutions.
For $d>0$ we consider
\begin{eqnarray}
M_d&=&\{x\in \R^N\ :\ \dist(x,M)\leq d\};\\
M_d^-&=&\{x\in M\ :\ \dist(x,M_\tau)\geq d\}.
\end{eqnarray}
We choose $d$ small enough such that
\begin{eqnarray}
&&G_\tau-\cat_{M_d} M_d=G_\tau-\cat_M M\\
&&G_\tau-\cat_M M_d^-=G_\tau-\cat_M (M-M_\tau)
\end{eqnarray}
Now we build two continuous operator
\begin{eqnarray}
\Phi_\eps^\tau&:&M_d^-\rightarrow {\cal N}_\eps^\tau\cap J_\eps^{2(m_\infty+\delta)};\\
\beta&:&{\cal N}_\eps^\tau\cap J_\eps^{2(m_\infty+\delta)}\rightarrow M_d,
\end{eqnarray}
such that $\Phi_\eps^\tau(\tau q)=-\Phi_\eps^\tau(q)$, $\tau\beta(u)=\beta(-u)$ and
$\beta\circ\Phi_\eps^\tau$ is $G_\tau$ homotopic to the inclusion $M_d^-\rightarrow M_d$.

By equivariant category theory we obtain
\begin{multline}
G_\tau-\cat_M (M-M_\tau)=G_\tau-\cat(M_d^-\hookrightarrow M_d)=\\
=G_\tau-\cat \beta\circ\Phi_\eps^\tau
\leq \Z_2-\cat {\cal N}_\eps^\tau\cap J_\eps^{2(m_\infty+\delta)}
\end{multline}
\section{Technical lemmas}

First of all, we recall that there exists a unique positive spherically symmetric function $U\in H^1(\R^n)$
such that
\begin{equation}
-\Delta U+U=U^{p-1} \text{ in }\R^n
\end{equation}
It is well known that $U_\eps(x)=U\left(\frac x\eps\right)$ is a solution of
\begin{equation}
-\eps^2\Delta U_\eps+U_\eps=U_\eps^{p-1}\text{ in }\R^n.
\end{equation}

Secondly, let us introduce the exponential map $\exp:TM\rightarrow M$ defined on the tangent
bundle $TM$ of $M$ which is a $C^\infty$ map. Then, for $\rho$ sufficiently small
(smaller than the injectivity radius of $M$ and smaller than $d/2$), 
the Riemannian manifold $M$ has a special set of
charts $\{\exp_x:B(0,\rho)\rightarrow M\}$.
Throughout the paper we will use the following notation: $B_g(x,\rho)$ is the open ball in $M$
centered in $x$ with radius $\rho$ with respect to the distance given by the metric $g$.
Corresponding to this chart, by choosing an orthogonal coordinate system
$(x_1,\dots,x_n)\subset \R^n$ and identifying $T_xM$ with $\R^n$ for $x\in M$, we can define
a system of coordinates called {\em normal coordinates}.

Let $\chi_\rho$ be a smooth cut off function such that
\begin{eqnarray*}
\chi_\rho(z)=1&&\text{if }z\in B(0,\rho/2);\\
\chi_\rho(z)=0&&\text{if }z\in \R^n \smallsetminus B(0,\rho);\\
|\nabla \chi_\rho(z)|\leq2&&\text{for all } x.
\end{eqnarray*}
Fixed a point $q\in M$ and $\eps>0$, let us define the function $w_{\eps,q}(x)$ on $M$ as
\begin{equation}
w_{\eps,q}(x)=\left\{
\begin{array}{cl}
U_\eps(\exp_q^{-1}(x))\chi_\rho(\exp_q^{-1}(x))&\text{if }x\in B_g(q,\rho)\\
&\\
0&\text{otherwise}
\end{array}
\right.
\end{equation}
For each $\eps>0$ we can define a positive number $t(w_{\eps,q})$ such that
\begin{equation}
\Phi_\eps(q)=t(w_{\eps,q})w_{\eps,q}\in H^1_g(M)\cap {\cal N_\eps}\text{ for }q\in M.
\end{equation}
Namely, $t(w_{\eps,q})$ turns out to verify
\begin{equation}
t(w_{\eps,q})^{p-2}=\frac{\displaystyle\int_M \eps^2|\nabla_g w_{\eps,q}|^2+|w_{\eps,q}|^2d\mu_g}
{\displaystyle\int_M |w_{\eps,q}|^pd\mu_g}
\end{equation}

\begin{lemma}
Given $\eps>0$ the application $\Phi_\eps(q):M\rightarrow H^1_g(M)\cap {\cal N}_\eps$
is continuous. Moreover, given $\delta>0$ there exists $\eps_0=\eps_0(\delta)$
such that, if $\eps<\eps_0(\delta)$
then $\Phi_\eps(q)\in {\cal N}_\eps\cap J_\eps^{m_\infty+\delta}$.
\end{lemma}
For the proof see \cite[ Proposition 4.2]{BBM07}.

At this point, fixed a point $q\in M_d^-$, let us define the function
\begin{equation}
\Phi_\eps^\tau(q)=t(w_{\eps,q})w_{\eps,q}-t(w_{\eps,\tau q})w_{\eps,\tau q}
\end{equation}
\begin{lemma}\label{lemma5}
Given $\eps>0$ the application
$\Phi^\tau_\eps(q):M_d^-\rightarrow H^1_g(M)\cap {\cal N}^\tau_\eps$
is continuous. Moreover, given $\delta>0$ there exists $\eps_0=\eps_0(\delta)$
such that, if $\eps<\eps_0(\delta)$
then $\Phi^\tau_\eps(q)\in {\cal N}^\tau_\eps\cap J_\eps^{2(m_\infty+\delta)}$.
\end{lemma}
\begin{proof}
Since $U_\eps(z)\chi_\rho(z)$ is radially symmetric we set
$U_\eps(z)\chi_\rho(z)=\tilde U_\eps(|z|)$. We recall that
\begin{eqnarray*}
|\exp^{-1}_{\tau q}\tau x|&=&d_g(\tau x,\tau q)=d_g(x,q)=|\exp^{-1}_{q} x|;\\
|\exp^{-1}_{q}\tau x|&=&d_g(\tau x, q)=d_g(x,\tau q).
\end{eqnarray*}
We have
\begin{multline*}
\tau^* \Phi_\eps^\tau(q)(x)=
-t(w_{\eps,q})w_{\eps,q}(\tau x)+t(w_{\eps,\tau q})w_{\eps,\tau q}(\tau x)=\\
=-t(w_{\eps,q})\tilde U_\eps(|\exp^{-1}_q(\tau x)|)+t(w_{\eps,\tau q})
\tilde U_\eps(|\exp^{-1}_{\tau q}(\tau x)|)=\\
=t(w_{\eps,\tau q})\tilde U_\eps(|\exp^{-1}_q(x)|)-t(w_{\eps,q})
\tilde U_\eps(|\exp^{-1}_{q}(\tau x)|)=\\
=t(w_{\eps,q})\tilde U_\eps(|\exp^{-1}_q(x)|)-
t(w_{\eps,q})\tilde U_\eps(|\exp^{-1}_{\tau q}(x)|),
\end{multline*}
because by the definition we have $t(w_{\eps,q})=t(w_{\eps,\tau q})$.

Moreover by definition the support of the function $\Phi_\eps^\tau$ is
$B_g(q,\rho)\cup B_g(\tau q,\rho)$, and $B_g(q,\rho)\cap B_g(\tau q,\rho)=\emptyset$ 
because $\rho<d/2$ and $q\in M_d^-$.
Finally, because
\begin{eqnarray}
&&\int_M|w_{\eps,q}|^\alpha d\mu_g=\int_M|w_{\eps,\tau q}|^\alpha d\mu_g\ \text{ for }\alpha=2,p;\\
&&\int_M|\nabla w_{\eps,q}|^2 d\mu_g=\int_M|\nabla w_{\eps,\tau q}|^2 d\mu_g,
\end{eqnarray}
we have
\begin{equation}
J_\eps(\Phi_\eps^\tau(q))=\left(\frac 12-\frac 1p\right)\frac 1{\eps^n}
\int_M |\Phi_\eps^\tau(q)|^pd\mu_g=2J_\eps(\Phi_\eps(q)).
\end{equation}
Then by previous lemma we have the claim.
\end{proof}
\begin{lemma}
We have $\displaystyle\lim_{\eps\rightarrow0}m_\eps^\tau=2m_\infty$
\end{lemma}
\begin{proof}
By the previous lemma and by Remark \ref{remps}
we have that for any $\delta>0$ there exists
$\eps_0(\delta)$ such that, for $\eps<\eps_0(\delta)$
\begin{equation}
2m_\eps\leq m_\eps^\tau\leq 2J_\eps(\Phi_\eps(q))\leq 2(m_\infty+\delta).
\end{equation}
Since $\displaystyle\lim_{\eps\rightarrow0}m_\eps=m_\infty$ (see \cite[Remark 5.9]{BBM07})
we get the claim.
\end{proof}

For any function $u\in {\cal N}_\eps^\tau$ we can define a point $\beta(u)\in \R^N$ by
\begin{equation}
\beta(u)=\frac{\displaystyle \int_M x|u^+(x)|^pd\mu_g}{\displaystyle \int_M |u^+(x)|^pd\mu_g}
\end{equation}
\begin{lemma}\label{lemma7}
There exists $\delta_0$ such that, for any $0<\delta<\delta_0$ and any $0<\eps<\eps_0(\delta)$
(as in Lemma \ref{lemma5}) and for any function
$u\in {\cal N}_\eps^\tau\cap J_\eps^{2(m_\infty+\delta)}$, it holds $\beta(u)\in M_d$.
\end{lemma}
\begin{proof}
Since $\tau^* u=u$ we set
\begin{eqnarray*}
M^+=\{x\in M\ :\ u(x)>0\}&&M^-=\{x\in M\ :\ u(x)<0\}.
\end{eqnarray*}
It is easy to see that $\tau M^+=M^-$. Then we have
\begin{multline}
J_\eps(u)=\left(\frac 12-\frac 1p\right)\frac 1{\eps^n}
\int_M|u|^pd\mu_g=\\
=\left(\frac 12-\frac 1p\right)\frac 1{\eps^n}
\left[\int_{M^+}|u^+|^pd\mu_g+\int_{M^-}|u^-|^pd\mu_g
\right]=2J_\eps(u^+)
\end{multline}
By the assumption $J_\eps(u)\leq 2(m_\infty+\delta)$ we have $J_\eps(u^+)\leq m_\infty+\delta$
then by Proposition 5.10 of \cite{BBM07} we get the claim.
\end{proof}
\begin{lemma}
There exists $\eps_0>0$ such that for any $0<\eps<\eps_0$ the composition
\begin{equation}
I_\eps=\beta\circ\Phi_\eps^\tau:M_d^-\rightarrow M_d\subset\R^N
\end{equation}
is well defined, continuous, homotopic to the identity and $I_\eps(\tau q)= \tau I_\eps(q)$.
\end{lemma}
\begin{proof}
It is easy to check that
\begin{eqnarray}
\Phi_\eps^\tau(\tau q)=-\Phi_\eps^\tau(q)&&\beta(-u)=\tau\beta(u).
\end{eqnarray}
Moreover, by Lemma \ref{lemma5} and by Lemma \ref{lemma7}, for any
$q\in M_d^-$ we have $\beta\circ\Phi_\eps^\tau(q)=\beta(\Phi_\eps(q))\in M_d$,
and $I_\eps$ is well defined.

In order to show that $I_\eps$ is homotopic to identity,
we evaluate the difference between $I_\eps$ and the identity as follows.
\begin{multline}
I_\eps(q)-q=\frac{\displaystyle\int_M(x-q)|u^+|^pd\mu_g}
{\displaystyle\int_M|u^+|^pd\mu_g}=
\frac{\displaystyle\int_{B(0,\rho)}
z\left|U\left(\frac z\eps\right)\chi_\rho(|z|)\right|^p
\big|g_q(z)\big|^{\frac12}}
{\displaystyle \int_{B(0,\rho)}
\left|U\left(\frac z\eps\right)\chi_\rho(|z|)\right|^p
\big|g_q(z)\big|^{\frac12}}=\\
=
\frac{\displaystyle \eps\int_{B(0,\frac {\rho}{\eps})}z
\big|U(z)\chi_\rho(|\eps z|)\big|^p
\big|g_q(\eps z)\big|^{\frac12}}
{\displaystyle \int_{B(0,\frac{\rho}{\eps})}
\big|U(z)\chi_\rho(|\eps z|)\big|^p
\big|g_q(\eps z)\big|^{\frac12}},
\end{multline}
hence $|I_\eps(q)-q| <\eps c(M)$ for a constant $c(M)$ that does not depend
on $q$.
\end{proof}

Now, by previous lemma and by Theorem \ref{castroclapp} we can prove Theorem \ref{mainteo}.

In fact, we know that, if $\eps$ is small enough, there exist
$G_\tau-\cat(M-M_\tau)$
minimizers which change sign,
because they are antisymmetric.
We have only to prove that
any minimizer changes sign exactly once.
Let us call $\omega=\omega_\eps$ one of these minimizers.
Suppose that the set $\{x\in M\ :\ \omega_\eps(x)>0 \}$ has $k$ connected
components $M_1, \dots, M_k$. Set
\begin{equation}
\omega_i=\left\{
\begin{array}{lll}
\omega_\eps(x)&&  x\in M_i\cup\tau M_i;\\
0&&\text{elsewhere}
\end{array}
\right.
\end{equation}
For all $i$, $\omega_i\in {\cal N}^\tau_\eps$.
Furthermore we have
\begin{equation}
J_\eps(\omega)=\sum\limits_iJ_\eps(\omega_i),
\end{equation}
thus
\begin{equation}
m_\eps^\tau=J_\eps(\omega)=\sum\limits_{i=1}^k J_\eps(\omega_i)\geq k \cdot m_\eps^\tau,
\end{equation}
so $k=1$, that concludes the proof.

\nocite{BC91}


\end{document}